\documentclass[12pt,reqno]{amsart}
\input{xypic}
\usepackage{amsmath}
\usepackage{amssymb}
\usepackage{latexsym}
\usepackage{enumerate}
\usepackage[ansinew]{inputenc}
\usepackage[all]{xy}
\usepackage[usenames]{color}

\newtheorem{theorem}{Theorem}[section]
\newtheorem{proposition}[theorem]{Proposition}
\newtheorem{corollary}[theorem]{Corollary}
\newtheorem{lemma}[theorem]{Lemma}
\newtheorem{example}[theorem]{Example}
\newtheorem{remark}[theorem]{Remark}
\newtheorem{question}[theorem]{Question}

\newcommand{\downarrowright}[1]{\downarrow
\rlap{\raise0.1cm\hbox{$\scriptstyle{#1}$}}}

\newcommand{\downarrowleft}[1]{\rlap{\kern-0.2cm
\raise0.1cm\hbox{$\scriptstyle{#1}$}}\downarrow}

\newcommand{\uparrowright}[1]{\uparrow
\rlap{\lower0.1cm\hbox{$\scriptstyle{#1}$}}}

\newcommand{\uparrowleft}[1]{\rlap{\kern-0.2cm
\lower0.1cm\hbox{$\scriptstyle{#1}$}}\uparrow}

\begin{document}
\setcounter{page}{1}
\title{Cellular covers of local groups}
\author[R. Flores]{Ram\'on Flores}
\author[J. Scherer]{J\'er\^{o}me Scherer}


\thanks{This work was supported by the Spanish Ministry of Education and Science MEC-FEDER grant MTM2010-15831 and
MTM2016-80439-P, and the Andalusian government under grants FQM-213 and P07-FQM-2863}


\maketitle

\begin{abstract}
We prove that, in the category of groups, the composition of a cellularization and a localization functor need not be idempotent.
This provides a negative answer to a question of Emmanuel Dror Farjoun.
\end{abstract}

\section*{Introduction}
Cellularization and localization functors are idempotent functors that are respectively augmented and coaugmented. The
interest for a systematic study of such functors comes from Homotopy Theory, mainly through early work of Bousfield, \cite{MR57:17648},
and Farjoun's more recent book \cite{Farjoun95}.
The names themselves are designed so as to remind us of cell complexes and $p$-local spaces or modules.
In Group Theory, localization functors appear explicitly and in full generality in Casacuberta's \cite[Section~3]{MR1320986}, but many
examples, such as abelianization or localization at a set of primes are classical. Likewise, cellularization functors for groups appear relatively late
and are described for the first time in \cite{MR2003a:55024}, even though specific cellular constructions such as universal central extensions have
played an important role in Group Theory.

It is helpful to think about localization and cellularization as functors that act on the whole category of groups, transforming some groups
drastically, possibly killing many of them, and allowing us to focus on some special features such as torsion or divisibility. In \cite{MR2995665} the authors
studied the effect of iterating different cellularization functors on a given (finite) group, that is, looking at the ``orbit'' of a group under the action
of all possible cellularization functors. A similar approach in the coaugmented case is probably quite difficult, as shown by the work of Rodr\'{i}guez and Scevenels, \cite{MR1804680}.

 The aim of this note is to analyze the effect of iterating a cellularization and a localization in the category of groups. Once a group $G$ has
been transformed by a localization functor $G \rightarrow LG$, and after it has been modified by a cellularization functor $cell LG \rightarrow LG$,
it seems that the group $cell LG$ is frequently left unchanged by further application of this procedure. In this case $cell LG$ is a fixed point in the
category of groups for the composition $cell  \circ L$. As we recall
below in Section~\ref{sec:idem}, all known examples provide such fixed points and it is in fact a question
Farjoun asked in \cite[Question~3]{MR2454697} whether $cell \circ L$ and $L \circ cell$ are idempotent functors,
without specifying the underlying category.

The first author gave a negative answer to this question in the case of homotopy localization and cellularization functors of spaces, \cite{MR2891699}.
His counterexample cannot be adapted in the category of groups as it uses in a central way the flexibility of having homotopy groups in different dimensions.
We have been wondering since then how to attack this problem for groups, one major obstruction being the difficulty to perform explicit computations.
Our theorem is based on a very recent computation, \cite{FloresMuro}, of a certain cellularization of large Burnside groups. The specific form of this
cellularization is the key to the unexpected behaviour of the iteration of the two functors we choose. It is rooted in the work of many mathematicians
who provided negative answers to the Burnside problem, in particular Ol'shanski\u\i's intriguing computation of the Schur multiplier
of Burnside groups at large primes, \cite[Corollary~31.2]{MR1191619}, see also Adian and Atabekyan's improved bound in \cite{Central}.

We conclude this introduction by recalling that any group theoretical example involving cellularization and localization can be upgraded
to a homotopical example via the classifying space functor. Thus, our counterexample provides a new counterexample to Farjoun's question
in the category of spaces, simpler than the original proof in \cite{MR2891699} in that all spaces are $K(\pi, 1)$'s, but more subtle in that
it relies on the restricted Burnside problem.

\setcounter{equation}{0}

\textbf{Acknowledgments.} We thank Antonio Viruel, Delaram Kahrobaei and Simon Smith for helpful conversations. We also thank
Varujan Atabekyan and Alexander Yu. Olshanskii for helping us out with the second homology group of Burnside groups.
The first author wishes to thank the \'{E}cole Polytechnique F\'{e}d\'{e}rale de Lausanne for its kind hospitality when this joint project started.

\section{Background on localization and cellularization}
\label{sec:back}
We will work mostly in the category of groups and will thus simply refer to localization and cellularization functors without specifying
that they are functors of groups. If necessary, we will make clear when we deal with space valued functors by calling them homotopical localization
and homotopical cellularization functors.

A \emph{localization} functor $L$ is an idempotent and coaugmented functor. The coaugmentation is traditionally written $\eta_G\colon G \rightarrow LG$
and idempotency means that both $\eta_{LG}$ and $L \eta_G$ are isomorphisms for any group $G$. Typical localization functors are obtained by
``inverting'' a group homomorphism $\varphi\colon A \rightarrow B$. A group $G$ is $\varphi$-\emph{local} if $\hbox{\rm Hom}(\varphi, G)$ is an isomorphism and
a group homomorphism $\psi$ is a $\varphi$-\emph{local equivalence} if $\hbox{\rm Hom}(\psi, G)$ is an isomorphism for any $\varphi$-local group~$G$.
Localization with respect to $\varphi$ is the localization functor $L_\varphi$ which is characterized by the fact that the coaugmentation
$\eta_G\colon G \rightarrow L_\varphi G$ is a $\varphi$-local equivalence to a $\varphi$-local group.

\begin{example}
\label{ex:loc}
{\rm Our main player in the world of localization functors is reduction mod $p$, i.e. localization with respect to the epimorphism
$\varphi\colon \mathbb Z \rightarrow \mathbb Z/p$, which, with the multiplicative notation, is the morphism $\varphi\colon F_1 \rightarrow C_p$
 from the free group on one generator to the cyclic group of order $p$, sending the  generator to a chosen generator.
We will write $L_p$ for this functor from now on. Loosely speaking, the effect of $L_p$ on a group $G$ is to kill $q$-torsion for $q \neq p$ and to
convert elements of order $p^n$ or of infinite order into order $p$ elements. In other words $L_p$ is left adjoint to the inclusion
in the category of groups of the variety of groups of exponent~$p$.
}
\end{example}

A \emph{cellularization} functor $cell$ is an idempotent and augmented functor. The augmentation is traditionally written $\epsilon_G\colon cell G \rightarrow G$.
The only known cellularization functors arise as follows. Choose  a group $A$ and define a group homomorphism $\psi$ to be an $A$-\emph{cellular
equivalence} if $\hbox{\rm Hom}(A, \psi)$ is an isomorphism. A group $G$ is $A$-\emph{cellular} if $\hbox{\rm Hom}(G, \psi)$ is an isomorphism for any
$A$-cellular equivalence $\psi$. Cellular groups are characterized by the property that they belong to $\mathcal C(A)$,
the smallest class containing $A$ and closed under
isomorphisms and arbitrary colimits, \cite[Section~3]{MR2003a:55024}.

\begin{example}
\label{ex:cell}
{\rm Our object of interest in the world of cellularization functors is $C_p$-cellularization, which we will write $cell_p$ for short.
Loosely speaking $cell_p G \rightarrow G$ is the best approximation of $G$ that can be constructed out of cyclic groups of order~$p$.
}
\end{example}

We mention next the construction by Casacuberta and Descheemaeker of relative group completions, \cite[Theorem 3.2]{MR2125447}.
We will call it \emph{fiberwise group localization}, which uses the name of the homotopical analogue, \cite[Theorem~1.F.1]{Farjoun95}.

\begin{theorem}[Casacuberta--Descheemaeker]
\label{fgl}
Let $A \rightarrowtail B\twoheadrightarrow C$ be an extension of groups, and let $L$ be a localization functor.
Then there exists a commutative diagram
\[
\xymatrix{
A \ar@{>->}[r] \ar[d]^{f_1} & B \ar@{->>}[r] \ar[d]^{f_2} & C \ar@{=}[d] \\
LA \ar@{>->}[r] & E \ar@{->>}[r] & C}
\]
such that $f_1$ is $L$-localization and $f_2$ is an $L$-local equivalence.
\end{theorem}

The methods we use in this article are mostly of a group theoretical nature. There are however situations where homotopy
theory provides helpful tools. The first one is about so-called \emph{nullification}
functors, i.e. localization with respect to a map of the form $A \rightarrow *$. The standard notation for such a functor is~$P_A$.
In group theory, the morphism $G \rightarrow P_A G$ is an epimorphism to the largest quotient of $G$ such that $\hbox{\rm Hom}(A, G)$
is trivial, \cite[Theorem~3.2]{MR1320986}. In homotopy theory $A$ is a space and $X \rightarrow P_A X$ is obtained from $X$,
up to homotopy, by taking the homotopy cofiber  $X_1$ of the map $\bigvee \Sigma^k A \rightarrow X$, where the wedge is taken
over representatives of all homotopy classes of maps out of suspensions of $A$, for $k \geq 0$, and by repeating this procedure.
The nullification $P_A X$ is the homotopy colimit of the telescope $X=X_0 \rightarrow X_1 \rightarrow X_2 \dots$, see for example
\cite[Section~17]{MR97i:55023}.
We will use the relationship between the group theoretical nullification with respect to $C_p$
and the homotopical nulllification with respect to the Moore space $M(C_p, 1) = S^1 \cup_p e^2$, as well as the analogous relation
between group theoretical and homotopical cellularization.

\begin{proposition}
\label{nullification}
Let $G$ be any group. Then
\begin{enumerate}
\item $\pi_1 P_{M(C_p, 1)} K(G, 1) \cong P_{C_p} G$;
\item $\pi_1 cell_{M(C_p, 1)} K(G, 1) \cong cell_{p} G$.
\end{enumerate}
\end{proposition}

The first isomorphism is \cite[Theorem~3.5]{MR1320986} and the second one is \cite[Theorem~3.3]{MR2003a:55024}.
In order to identify the space $cell_{M(C_p, 1)} K(G, 1)$ we will
use Chach\'olski's fibration from \cite[Theorem~20.3]{MR97i:55023}. We start with the cofibration sequence
\[
\bigvee M(C_p, 1) \rightarrow K(B, 1) \rightarrow Cof
\]
where the wedge is taken over representatives of all homotopy classes from $M(C_p, 1)$ to $K(B, 1)$. The Moore
space $M(C_p, 2)$ is the suspension of $M(C_p, 1)$

\begin{theorem}[Chach\'olski]
\label{wojtek}
The cellularization $cell_{M(C_p, 1)} K(G, 1)$
is weakly equivalent to the homotopy fiber of the composite map $K(G, 1) \rightarrow Cof \rightarrow P_{M(C_p, 2)} Cof$.
\end{theorem}

\section{Idempotent examples}
\label{sec:idem}
This section is devoted to describe examples where the composition of a cellularization and a localization functor is an idempotent functor. Note that the composite will not be, in general, augmented nor coaugmented.

\begin{example}
\label{ex:PCW}
{\rm Let $p$ be a prime number. Consider the cellularization functor $cell_p$ introduced above, and $P_p$ is nullification with respect to $C_p$,
i.e. localization with respect to the constant map $C_p\rightarrow \{1\}$. The effect of $cell_p$ is described in \cite[Theorem~3.7]{MR2003a:55024},
while, for any group $G$, $P_p G$ is isomorphic to the quotient of $G$ by its $p$-radical $T_p(G)$, which is the largest $p$-torsion free quotient of~$G$.

The cellularization $cell_p G$ belongs to the class of $C_p$-cellular groups $\mathcal C(C_p)$, which is contained in the class $\overline{\mathcal C(C_p)}$
of $C_p$-acyclic groups, \cite[Section~4]{MR2003a:55024} (this class consists by definition in all groups $G$ that are ``killed'' by $C_p$, i.e. $P_p G = \{ 1 \}$).
In particular $P_p cell_p G = \{1 \}$.
Likewise, as there are no non-trivial homomorphisms from $C_p$ to $P_p G$, the group $cell_p P_p G$ is also trivial. So in this case both functors
$P_p cell_p$ and $cell_p P_p$ are idempotent.
}
\end{example}

\begin{example}
\label{ex:PCW2}
{\rm Consider now two distinct primes $p$ and $q$, and assume furthermore that $G$ is finite. In this case $cell_p G$ has been described in
\cite[Lemma~4.5]{MR2272149} as an extension
\[
K\rightarrow cell_pG \rightarrow S_pG,
\]
where $S_pG$ is the subgroup of $G$ generated
by its order $p$ elements, and $K$ is the quotient of $H_2(S_pG; \mathbb Z)$ by its  $p$-torsion subgroup. We compute $cell_p P_qG$. As $P_q G=G/T_qG$,
$q$ does not divide the order of $P_q G$, and the same happens to $cell_p P_qG$ as can be seen from the previous extension.
Thus, there is an isomorphism $cell_p P_qG \cong P_q cell_p P_qG$, so that the latter is a $C_p$-cellular group.
This implies that $cell_p P_q$ is idempotent when applied to $G$.

This is an instance of a very frequent phenomenon,
in which some cellular deformation of a local group turns out to be local again, or, as illustrated in our next example,
some localization of a cellular group is cellular. Both cases imply idempotency in the composition.
}
\end{example}

\begin{example}
\label{ex:PCW3}
{\rm We wish here to understand $P_q cell_p$ for two distinct primes $p$ and~$q$. Let $G$ be a finite group.
We will prove that $P_q cell_p G$ is cellular.

From the extension in Example~\ref{ex:PCW2}, as $G$ is finite, it follows that we only need to check that $P_q cell_p G$ is generated by order $p$ elements,
and that $H_2 (P_q cell_p G; \mathbb Z)$ is $p$-torsion. The former is clear, as $P_q cell_p G$ is a quotient of $cell_p G$, a finite
group generated by elements of order~$p$.

To prove the latter, observe that $P_q cell_p G$ is $q$-torsion free, hence so is the second homology group $H_2(P_q cell_p G; \mathbb Z)$.
We use now Proposition~\ref{nullification}~(1) to identify $P_q cell_p G$ with the fundamental group of $P_{M(C_q, 1)} K(cell_p G, 1)$.
This space is obtained by killing successively all maps out of the Moore space $M(C_q, 1)$ and its suspensions. The homology groups of the
homotopy cofibers $X_0, X_1, X_2$, etc., described above fit thus in an exact sequence
\[
H_2(X_k; \mathbb Z) \rightarrow H_2(X_{k+1}; \mathbb Z) \rightarrow H_1(\bigvee \Sigma^k M(C_q, 1); \mathbb Z)\, .
\]
Since $H_2(X_0; \mathbb Z)$ is $p$-torsion -- because $X_0 = K(cell_p G, 1)$ -- and the homology of the Moore space $M(C_q, 1)$ is $q$-torsion,
this proves by induction that the second homology group $H_2 (P_{M(C_q, 1)} K(cell_p G, 1); \mathbb Z)$ is an abelian group having only $p$- and $q$-torsion.

Let us call $X = P_{M(C_q, 1)} K(cell_p G, 1)$, a space whose fundamental group is isomorphic to $P_q cell_p G$ and consider the
universal covering fibration sequence $$\tilde X \rightarrow X \rightarrow K(P_q cell_p G, 1).$$ An easy Serre spectral sequence
argument shows that $H_2(P_q cell_p G; \mathbb Z)$ is a quotient of $H_2(X; \mathbb Z)$ (this phenomenon had been already observed by Hopf, \cite{MR0006510}).
This second homology group must thus be $p$-torsion, which shows that $P_q cell_p G$ is $C_p$-cellular. Hence $P_q cell_p P_q cell_p G = P_q cell_p G$.
}
\end{example}

\begin{example}
\label{ex:idempotent2}
{\rm Often $L cell G$ is the trivial group. This happened in Example~\ref{ex:PCW}, and also for example when $G$ is a finite simple group,
$L=L_{ab}$ is abelianization (localization with respect to the homomorphism $\mathbb Z * \mathbb Z \rightarrow \mathbb Z \oplus \mathbb Z$)
and $cell G$ is any cellular cover of $G$. This comes from the fact that $cell G$ is a perfect group, \cite[Section~11]{MR2995665}.
}
\end{example}

\begin{example}
\label{ex:idempotent}
{\rm Let $L=L_{ab}$ be abelianization.
Since any cellularization functor transforms an abelian group into an abelian group, \cite[Theorem~1.4]{MR2269828},
$cell L_{ab} G$ is an abelian group which is thus both cellular and $L_{ab}$-local.
}
\end{example}


\begin{example}
\label{ex:idempotent3}
{\rm The inclusion $i\colon A_5 \hookrightarrow A_6$ is a localization by \cite[Section~3(i)]{MR1942308}. The group $A_5$ is $C_2$-cellular since
it is generated by elements of order $2$ and $H_2(A_5; \mathbb Z) \cong \mathbb Z/2$.

We choose thus $cell_2$ and $L_i$ so that $L_i cell_2 A_5 = A_6$. However $A_6$ is not $C_2$-cellular because $H_2(A_6; \mathbb Z) \cong \mathbb Z/6$.
In fact $cell_2 A_6$ must be a central extension with prime to $p$ torsion center, so it is $3.A_6$, the Valentiner group, a triple cover of $A_6$.
This group is not $i$-local as it contains $A_5$, but not $A_6$ (the central extension is not split). It is very likely that $L_i (3.A_6)$ is $A_6$, but
we have not found a proof this fact.
}
\end{example}

\section{Cellularizing the Burnside group}
Let $p$ be a prime and let $B$ be the\emph{ Burnside group} $B_{2,p}$, which is the quotient of a free group $F_2$ on two generators
by the normal subgroup generated by all $p$-th powers. Thus $B=L_p F_2$ is free in the variety of groups of exponent $p$. Alternatively,
and this is our starting point for the next computations, let $F$ be the free product $C_p * C_p$ of two copies of cyclic groups of order $p$.
This is a $C_p$-cellular group and $B \cong L_p F$ as $F$ is obtained by adding to the free group $F_2$ only relations which are exponent $p$ reduction.
We have thus $L_p cell_p F \cong B$.

The computation we present in this section is that of $cell_p B$, for a sufficiently large prime. By this we mean that $p\geq 665$, i.e.
$p$ is a prime at least equal to $671$. We rely on Ol'shanski\u\i's work
in \cite{MR1191619} where $p>10^{10}$, or rather on the improved bound $p\geq 665$ in Adian and Atabekyan's \cite{Central}.
Our next proposition is a particular case of the much more general statement
\cite[Proposition~3.9]{FloresMuro}. We identify $cell_p B$  by elementary methods for the sake of completeness, using in a crucial
way the strong link between homotopical and group theoretical cellularization.

\begin{proposition}\cite[Proposition~5.2]{FloresMuro}
\label{cellB}
Let $p$ be sufficiently large. The $C_p$-cellular approximation of $B$ fits in a central extension
\[
K \rightarrowtail cell_p B \twoheadrightarrow B
\]
where $K$ is isomorphic to $H_2(B; \mathbb Z)$, the Schur multiplier of $B$, a free abelian group in a countable number of generators.
\end{proposition}

\begin{proof}
We have seen in Proposition~\ref{nullification}~(2) that
$cell_p B$ is isomorphic to the fundamental group of $cell_{M(C_p, 1)} K(B, 1)$. Our goal is thus to compute this cellularization of
the Eilenberg--Mac Lane space $K(B, 1)$ (in the category of pointed spaces). For this purpose we use Chach\'olski's Theorem~\ref{wojtek}
and start with the cofibration sequence
\[
\bigvee M(C_p, 1) \rightarrow K(B, 1) \rightarrow Cof
\]
where the wedge is taken over representatives of all homotopy classes from $M(C_p, 1)$ to $K(B, 1)$.
The cellularization $cell_{M(C_p, 1)} K(B, 1)$
is then the homotopy fiber of the composite map $K(B, 1) \rightarrow Cof \rightarrow P_{M(C_p, 2)} Cof$.

Since $B$ is generated by elements of order $p$, the cofiber $Cof$ is simply connected and $H_2(Cof; \mathbb Z)$ is an extension of
a direct sum $\oplus \mathbb Z/p$ by $H_2(B; \mathbb Z)$. The latter is a free abelian group in a countable number of generators
by \cite[Corollary~31.2]{MR1191619} or \cite{Central}, this is where we use that $p$ is large. When nullifying with respect to the simply connected space
$M(C_p, 2) = S^2 \cup_p e^3$ one kills the $p$-torsion in $H_2(Cof; \mathbb Z)$ so that only a free abelian group is left in
$H_2(P_{M(C_p, 2)} Cof; \mathbb Z) \cong \pi_2 P_{M(C_p, 2)} Cof$. The homotopy long exact sequence associated to Chach\'olski's
fibration sequence $cell_{M(C_p, 1)} K(B, 1) \rightarrow K(B, 1) \rightarrow P_{M(C_p, 2)} Cof$ yields an extension
\[
\pi_2 P_{M(C_p, 2)} Cof \rightarrowtail \pi_1 cell_{M(C_p, 1)} K(B, 1) \twoheadrightarrow B
\]
which allows us to conclude.
\end{proof}

\begin{remark}
\label{rem:smallp}
{\rm When $p=2$, the Burnside group $B = C_2 \times C_2$ is $C_2$-cellular, i.e. $cell_2 B = B$. At the prime $3$,
$B$ is an extraspecial nilpotent group of class $2$ of order $3^3$, hence $C_3$-cellular since $H_2(B; \mathbb Z)$ is $3$-torsion,
see \cite[Proposition~4.8]{MR2272149}. At the prime $5$ it is not known whether $B$ is finite and we do not know what $cell_p B$
looks like at a ``small'' prime $5 \leq p < 671$.}
\end{remark}

\section{Localizing the cellularized Burnside group}










From now on we will denote $cell_p B$ by $C$ and $p$ is a sufficiently large prime.
In this section we compute $L_p C$ and the next lemma will help us on the way.

\begin{lemma}
\label{central}
The fiberwise $L_p$-localization of the extension $C$ is a central extension $L_p K \rightarrowtail E \twoheadrightarrow B$.
\end{lemma}

\begin{proof}
If we apply fibrewise group localization (Theorem~\ref{fgl}) to the extension in Proposition~\ref{cellB},
we obtain a commutative diagram of homomorphisms:
\[
\xymatrix{
K \ar@{>->}[r] \ar[d]^{\eta} & C \ar@{->>}[r]^\epsilon \ar[d]^{f} & B \ar@{=}[d] \\
L_p K \ar@{>->}[r] & E \ar@{->>}[r]^\pi & B}
\]
We observe that the coaugmentation $\eta: K \rightarrow L_p K$ is exponent $p$ reduction, i.e. $L_p K \cong \oplus \mathbb Z/p$ is a
countable $\mathbb{F}_p$-vector space. We deduce then from the Five Lemma that $f: C \rightarrow E$ is an epimorphism, which sends
by naturality the center of $C$ to the center of $E$. The Burnside group $B$ is known to be centerless  by \cite[Theorem~VI.3.4]{MR537580},
so that $Z(C) = K$. Therefore $L_p K$ is contained in $Z(E)$, in other words the extension
$L_p K \rightarrowtail E \twoheadrightarrow B$ is central.
\end{proof}

\begin{proposition}
\label{Eexpp}
Let $E$ be the fiberwise $L_p$-localization of the extension $C$. Then $E$ is a group of exponent $p$.
\end{proposition}

\begin{proof}
As $L_f K$ and $B$ are groups of exponent $p$, $E$ is a group whose non-trivial elements can have order $p$ and $p^2$.
We should check that the latter is impossible.
Let $x$ be any element in $E$ and set $b=\pi(x)$. As $C=cell_p B$ is the $C_p$-cellularization of $B$, the augmentation induces a bijective map
\[
\epsilon_*\colon \textrm{Hom}(C_p,C) \cong \textrm{Hom}(C_p,B)
\]
and this guarantees that there exists a preimage $c \in \epsilon^{-1}(b)$ of order $p$ (or one if $b = 1$).
By commutativity of the diagram in the proof of Lemma~\ref{central}, $\pi(x)=\epsilon(c)=\pi(f(c))$, and hence $x=f(c)t$ for a certain element $t\in L_p K$.
Since $t$ belongs to $L_f K$, a group of exponent~$p$, both $c$ and $t$ have order (at most) $p$. Moreover $t$ is central in $E$
so that $x^p=(f(c)t)^p=f(c)^pt^p=1$, and we are done.
\end{proof}

\begin{corollary}
The extension $E$ is the $L_p$-localization of $C$.
\end{corollary}

\begin{proof}
Because of Theorem \ref{fgl}, the map $f:C\rightarrow E$ is an $L_p$-equivalence.
By the previous proposition, $E$ is $L_p$-local, and then $L_p C \cong E$.
\end{proof}

In fact, we can say more about the structure of $E$.

\begin{proposition}
\label{splitting}
The extension $E$ splits as a direct product $E\cong L_p K\times B$.
\end{proposition}

\begin{proof}
We will prove that the map $\pi\colon E \rightarrow B$ has a section. As the extension is central by Lemma~\ref{central}, it must
then split as a direct product.

The Burnside group $B$ is free, in the variety of groups of exponent $p$,
on two generators $\bar x$ and $\bar y$. Let us choose preimages $x$ and $y$ in $E$. Remember that $B$ is $L_p F_2$ and call $a$ and $b$
two generators of the free group $F_2$ that are taken by the coaugmentation $\eta\colon F_2 \rightarrow B$ to $\bar x$ and $\bar y$ respectively.

Since $F_2$ is free, there exists a unique lift $q\colon F_2 \rightarrow E$ (taking $a$ to $x$ and $b$ to $y$).
By Proposition~\ref{Eexpp} $E$ is a group of exponent $p$, hence $q$ factors uniquely through a homomorphism $g\colon B\rightarrow E$.
This means that the composite
\[
\pi \circ g \circ \eta: F_2 \rightarrow B \rightarrow E \rightarrow B
\]
coincides with the coaugmentation $\eta$.
The universal property of localization shows then that $\pi \circ g$ is the identity, i.e. $g$ is a section of $\pi$.
\end{proof}

\begin{remark}
{\rm Note that a similar splitting cannot exist for the extension that defines~$C$. If it did exist, this would imply that the free abelian group $K$
would be a retract of a $C_p$-cellular space, hence $C_p$-cellular itself.
}
\end{remark}

Now we have produced one of our two desired counterexamples.

\begin{theorem}
The groups $L_p cell_p F$ and $L_p cell_p L_p cell_p F$ are not isomorphic.
\end{theorem}

\begin{proof}
We have proved that $L_p cell_p F$ is the Burnside group $B$, which is centerless.
On the other hand, Proposition~\ref{splitting} implies that $L_p cell_p L_p cell_p F = E$ has (many) order $p$ elements in the center.
\end{proof}

\section{Cellularizing one step further}

The next result is an easy consequence of Proposition \ref{splitting}.

\begin{proposition}
\label{cellE}
We have an isomorphism $cell_p E \cong L_p K\times C$.
\end{proposition}

\begin{proof}
The cellularization of a product is the product of the cellularizations, see for example \cite[Lemma~2.3(1)]{MR2269828}.
We have then $cell_p E \cong cell_p (L_p K \times B) \cong L_p K \times C$ because $L_p K$ is $C_p$-cellular.
\end{proof}

From this we obtain our second counterexample.

\begin{theorem}
The groups $cell_p L_pB$ and $cell_p L_p cell_p L_p B$ are not isomorphic.
\end{theorem}

\begin{proof}
The Burnside group $B$ is $L_p$-local and $cell_p L_pB$ is the group $C$.
The group $cell_p L_p cell_p L_p B = cell_p L_p C = cell_p E$ is the direct product $L_p K \times C$ as we have established
in Proposition~\ref{cellE}. This group has uncountably many order $p$ elements in the center whereas the center of $C$ is the
free abelian group~$K$.
\end{proof}

We have been trying out quite a few localization and cellularization functors before coming up with the above computation.
Let us mention a few promising ones, which we have not been able to fully understand. Libman showed that the inclusion
$A_n \hookrightarrow A_{n+1}$ is a localization (with respect to the inclusion itself), \cite[Example~3.4]{MR1760593}, and the
universal central extension $2.A_{n+1} \twoheadrightarrow A_{n+1}$ is a cellular cover by \cite[Section~11]{MR2995665}. Libman
also found examples of infinite localizations of finite groups, \cite{MR1758733}, and Prze\'zdziecki showed they can be arbitrarily
large, \cite{MR2464105}. Iterating this procedure is not manageable, but if we wish to understand the dynamics of the effect
of idempotent functors on the category of groups it would be important to make progress in this direction.

To conclude we would like to address a natural question.

\begin{question}
{\rm  Is there a cellularization functor $cell$, a localization functor $L$, and an abelian group $A$ such that $L cell L cell A \not\cong L cell A$ or
$ cell L cell L A \not\cong cell L A$?
}
\end{question}

We believe that the answer to Farjoun's question is negative as well in the category of abelian groups.
It is tempting to use the inclusion $\bigoplus_n \mathbb Z/p^n \hookrightarrow \prod_n \mathbb Z/p^n$, which is a localization by
\cite[Example~6.5]{MR2127962}, but we have not found an adequate cellularization functor to continue.


\bibliographystyle{amsplain}
\providecommand{\bysame}{\leavevmode\hbox to3em{\hrulefill}\thinspace}
\providecommand{\MR}{\relax\ifhmode\unskip\space\fi MR }
\providecommand{\MRhref}[2]{%
  \href{http://www.ams.org/mathscinet-getitem?mr=#1}{#2}
}
\providecommand{\href}[2]{#2}


\medskip
\begin{minipage}[t]{8 cm}
Ram\' on Flores\\
Departamento de Geometr\'{\i}a y Topolog\'{\i}a\\
Universidad de Sevilla\\
41080 - Sevilla, Spain\\
\textit{E-mail:}\texttt{\, ramonjflores@us.es}\\
\end{minipage}
\begin{minipage}[t]{8 cm}
J\'er\^ome Scherer\\
Institute of Mathematics\\
EPFL, Station 8\\
CH - 1015 Lausanne, Switzerland\\
\textit{E-mail:}\texttt{\,jerome.scherer@epfl.ch}\\
\end{minipage}

\end{document}